%% file: paper-v6.tex
\documentclass[twoside,11pt]{article}

\usepackage{graphicx, subfigure}
\usepackage[top=1in,bottom=1in,left=1in,right=1in]{geometry}
\usepackage[sort&compress, numbers]{natbib} 
\usepackage{amsfonts, amsmath, amssymb, amsthm, constants, bbm}
\usepackage{MnSymbol}
\usepackage{enumitem}

\usepackage{color}
\definecolor{darkred}{RGB}{100,0,0}
\definecolor{darkgreen}{RGB}{0,100,0}
\definecolor{darkblue}{RGB}{0,0,150}

\usepackage{hyperref}
\hypersetup{colorlinks=true, linkcolor=darkred, citecolor=darkgreen, urlcolor=darkblue}
\usepackage{url}

\input notation.tex

\definecolor{purple}{rgb}{0.4,.1,.9}

\newcommand\blfootnote[1]{%
  \begingroup
  \renewcommand\thefootnote{}\footnote{#1}%
  \addtocounter{footnote}{-1}%
  \endgroup
}

\begin{document}
\thispagestyle{empty}

\title{The Sparse Variance Contamination Model}
\author{Ery Arias-Castro \and Rong Huang}
\date{}
\maketitle

\begin{abstract}
We consider a Gaussian contamination (i.e., mixture) model where the contamination manifests itself as a change in variance.  We study this model in various asymptotic regimes, in parallel with the work of Ingster (1997) and Donoho and Jin (2004), who considered a similar model where the contamination was in the mean instead.
\end{abstract}

\blfootnote{Both authors are with the Department of Mathematics, University of California, San Diego (\url{math.ucsd.edu}).}

\section{Introduction} \label{sec:intro}

The detection of rare effects becomes important in settings where a small proportion of a population may be affected by a given treatment, for example.  The situation is typically formalized as a contamination model.  Although such models have a long history (e.g., in the theory of robust statistics), we adopt the perspective of \citet{ingster1997some} and \citet{donoho2004higher}, who consider such models in asymptotic regimes where the contamination proportion tends to zero at various rates.  This line of work has mostly focused on models where the effect is a shift in mean, with some rare exceptions \cite{cai2014optimal, cai2011optimal}.  In this paper, instead, we model the effect as a change in variance.  

We consider the following contamination model:
\beq\label{model}
(1 - \eps) \cN(0, 1) + \eps \cN(0, \sigma^2),
\eeq
where $\eps \in [0,1/2)$ is the contamination proportion and $\sigma > 0$ is the standard deviation of the contaminated component.  (Note that this is a Gaussian mixture model with two components.)  Following \cite{ingster1997some, donoho2004higher}, we consider the following hypothesis testing problem: based on $X_1, \dots, X_n$ drawn iid from \eqref{model}, decide
\beq\label{problem}
\cH_0: \eps = 0 \quad \text{versus} \quad \cH_1: \eps > 0,\, \sigma \ne 1.
\eeq

As usual, we study the behavior of the likelihood ratio test, which is optimal in this simple versus simple hypothesis testing problem.
We also study some testing procedures that, unlike the likelihood ratio test, do not require knowledge of the model parameters $(\eps, \sigma)$:
\bitem
\item The {\em chi-squared test} rejects for large values of $\big|\sum_i X_i^2 - n\big|$.
This is the typical variance test when the sample is known to be zero mean.
\item The {\em extremes test} rejects for combines the test that rejects for small values of $\min_i |X_i|$ and the test that rejects for large values of $\max_i |X_i|$ using Bonferroni's method.
\item The {\em higher criticism test} \cite{donoho2004higher} amounts to applying one of the tests proposed by \citet{anderson1952asymptotic} for normality. 
One variant is based on rejecting for large values of
\beq\label{HC}
\sup_{x \ge 0} \frac{\sqrt{n}\, |F_n(x) - \Psi(x)|}{\sqrt{\Psi(x)(1-\Psi(x))}},
\eeq
where $\Psi(x) := 2\Phi(x) - 1$, where $\Phi$ denotes the standard normal distribution, and $F_n(x) := \frac1n \sum_{i=1}^n \IND{|X_i| \le x}$.
\eitem

The testing problem \eqref{problem} was partially addressed by \citet*{cai2011optimal}, who consider a contamination model where the effect manifests itself as a shift in mean and a change in variance.  However, in their setting the variance is fixed, while we let the variance change with the sample size in an asymptotic analysis that is now standard in this literature. 

Our analysis reveals three distinct situations:
\renewcommand{\theenumi}{(\alph{enumi})}
\renewcommand{\labelenumi}{\theenumi}
\benum  \setlength{\itemsep}{0in}
\item {\em Near zero} ($\sigma \to 0$): In the sparse regime, the higher criticism test is as optimal as the likelihood ratio test, while the chi-squared test is powerless and the extremes test is suboptimal.
\item {\em Near one} ($\sigma \to 1$): In the dense regime, the chi-squared test and the higher criticism test are as optimal as the likelihood ratio test, while the extremes test has no power.
\item {\em Away from zero and one}: In the sparse regime, the extremes test and the higher criticism test are as optimal as the likelihood ratio test, while the chi-squared test is asymptotically powerless if $\sigma$ is bounded.
\eenum

In the tradition of \citet{ingster1997some}, we set
\beq\label{eps}
\eps = n^{-\beta}, \quad \beta \in (0, 1) \text{ fixed}.
\eeq
The setting where $\beta \le 1/2$ is often called the dense regime while the setting where $\beta > 1/2$ is often called the sparse regime.  (Note that the setting where $\beta > 1$ is uninteresting since in that case there is no contamination with probability tending to~1.)
%\beq
%\sqrt{n} \eps \to
%\begin{cases}
%\infty & \text{ dense regime; }\\
%0 & \text{ sparse regime. }
%\end{cases}
%\eeq
%In the dense regime, we are more interested in situation (b) and (c), while in the sparse regime, we are more interested in situation (a) and (d).

\section{The likelihood ratio test}

We start with bounding the performance of the likelihood ratio test.  As this is the most powerful test by the Neyman-Pearson Lemma, this bound also applies to any other test.  
We say that a testing procedure is asymptotically powerless if the sum of its probabilities of Type I and Type II errors (its risk) has limit inferior at least~1 in the large sample asymptote.

\subsection{Near zero}
Consider the testing problem \eqref{problem} in the regime where $\sigma = \sigma_n \to 0$ as $n\to\infty$.  More specifically, we adopt the following parameterization as it brings into focus the first-order asymptotics:
\beq\label{sigma0}
\sigma = n^{-\gamma}, \quad \gamma > 0 \text{ fixed}.
\eeq  

\begin{thm}\label{thm:LR0}
For the testing problem \eqref{problem} with parameterization \eqref{eps} and \eqref{sigma0}, the likelihood ratio test (and then any other test procedure) is asymptotically powerless when 
\beq\label{LR0}
\gamma < 2\beta -1.
\eeq
\end{thm}

\begin{proof}
The likelihood ratio is 
\beq
L := \prod_{i=1}^n L_i,
\eeq
where $L_i$ is the likelihood ratio for observation $X_i$, which in this case is
\begin{align}
L_i 
&= \frac{\frac{1 - \eps}{\sqrt{2\pi}} \exp(-\frac12 X_i^2) + \frac\eps{\sqrt{2\pi} \sigma} \exp(-\frac1{2\sigma^2} X_i^2)}{\frac1{\sqrt{2\pi}} \exp(-\frac12 X_i^2)} \\
&= 1 -\eps + \frac{\eps}{\sigma} \exp\bigg(\frac{\sigma^2 -1}{2 \sigma^2} X_i^2\bigg)
\end{align}

The risk of the likelihood ratio test is equal to
\beq
{\rm risk}(L) := 1 - \frac12 \E_0 |L - 1|.
\eeq
Our goal is to show that ${\rm risk}(L) = 1 + o(1)$ under the stated conditions.
When $\sigma$ is below and bounded away from $\sqrt{2}$, it turns out that a crude method, the so-called 2nd moment method which relies on the Cauchy-Schwarz Inequality, is enough to lower bound the risk.
Indeed, by the Cauchy-Schwarz Inequality,
\beq
{\rm risk}(L) \ge 1 - \frac12 \sqrt{\E_0[L^2] - 1},
\eeq
and we are left with the task of finding conditions under which $\E_0[L^2] \le 1 + o(1)$. 

We have
\beq
\E_0[L^2] 
= \prod_{i=1}^n \E_0[L_i^2]
= (\E_0[L_1^2])^n,
\eeq
where
\begin{align}
\E_0[L_1^2] 
&= \E_0\bigg[\bigg(1 -\eps + \frac{\eps}{\sigma} \exp\bigg(\frac{\sigma^2 -1}{2 \sigma^2} X_1^2\bigg)\bigg)^2\bigg]\\
& = 1 - \eps^2 + \frac{\eps^2}{\sigma^2} \E_0\bigg[\exp\bigg(\frac{\sigma^2 -1}{\sigma^2} X_1^2\bigg)\bigg]\\ 
& = 1 - \eps^2 + \eps^2 \big[\sigma^2(2-\sigma^2)\big]^{-1/2}\\
& = 1 + \eps^2 \Big(\big[\sigma^2(2-\sigma^2)\big]^{-1/2} - 1\Big).
\end{align}
Therefore, 
\beq
\E_0[L^2]
= \bigg[1 + \eps^2 \Big(\big[\sigma^2(2-\sigma^2)\big]^{-1/2} - 1\Big)\bigg]^n
\le \exp\bigg[n \eps^2 \Big(\big[\sigma^2(2-\sigma^2)\big]^{-1/2} - 1\Big)\bigg],
\eeq
so that $\E_0[L^2] \le 1 + o(1)$ when 
\beq\label{2nd-moment}
n \eps^2 \Big(\big[\sigma^2(2-\sigma^2)\big]^{-1/2} - 1\Big) \to 0.
\eeq
Plugging in the parameterization \eqref{eps} and \eqref{sigma0}, we immediately see that this condition is fulfilled when \eqref{LR0} holds, and this concludes the proof.
\end{proof}

%\begin{rem}[Between 0 and 1]
%The same arguments show that the likelihood ratio is asymptotically powerless when $\sigma \in (0,1)$ is fixed (or bounded away from 0 and 1) and $\beta > 1/2$.
%\end{rem}
%
%\begin{rem}[Away from 1] \label{rem:away1}
%It is a consequence of \prpref{chi2} that the likelihood ratio test is asymptotically powerful in the dense regime where $\beta < 1/2$ as soon as $\sigma$ is bounded away from~1.
%\end{rem}

\subsection{Near one}
Consider the testing problem \eqref{problem} in the regime where $\sigma^2 \to 1$.  More specifically, we adopt the following parameterization:
\beq\label{sigma1}
|\sigma - 1| = n^{-\gamma}, \quad \gamma > 0 \text{ fixed}.
\eeq  

\begin{thm}\label{thm:LR1}
For the testing problem \eqref{problem} with parameterization \eqref{eps} and \eqref{sigma1}, the likelihood ratio test (and then any other test procedure) is asymptotically powerless when 
\beq\label{LR1}
\gamma > 1/2 - \beta.
\eeq
\end{thm}

\begin{proof}
Restarting the proof of \thmref{LR0} at \eqref{2nd-moment}, and plugging in the parameterization \eqref{eps} and \eqref{sigma1}, we immediately see that $\E_0[L^2] \le 1 + o(1)$ when \eqref{LR1} holds.
\end{proof}

\subsection{Away from zero and one}
Consider the testing problem \eqref{problem} in the regime where $\sigma$ is fixed away from 0 and 1.  (Some of the results developed in this section are special cases of results in \cite{cai2011optimal}.)  

\begin{thm}
For the testing problem \eqref{problem} with parameterization \eqref{eps} and $\sigma > 0$ is fixed, the likelihood ratio test (and therefore any other test) is asymptotically powerless when $\beta > 1/2$ and 
\beq\label{LR1+}
\sigma < 1/\sqrt{1-\beta}.
\eeq
\end{thm}

\begin{proof}
We use a refinement of the second moment method, sometimes called the  truncated second moment method, which is based on bounding the moments of a thresholded version of the likelihood ratio.
Define the indicator variable $D_i = \IND{|X_i| \le \sqrt{2 \log n}}$ and the corresponding truncated likelihood ratio
\beq
\bar L = \prod_{i=1}^n \bar L_i, \quad \bar L_i := L_i D_i.
\eeq  

Using the triangle inequality, the fact that $\bar L\leq L$, and the Cauchy-Schwarz Inequality, we have the following upper bound:
\begin{align}
\E_0[|L-1|]
&\le \E_0[|\bar L-1|]+\E_0[L-\bar L] \\
&\le \big[ \E_0[\bar L^2] - 1 + 2 (1 - \E_0[\bar L]) \big]^{1/2} + (1-\E_0[\bar L]) \ ,
\end{align}
so that ${\rm risk}(L) \ge 1 + o(1)$ when $\E_0[\bar L^2] \le 1 + o(1)$ and $\E_0[\bar L] \ge 1 + o(1)$. 

For the first moment, we have
\beq
\E_0[\bar L] 
= \prod_{i=1}^n \E_0[\bar L_i]
= \E_0[\bar L_1]^n,
\eeq
so that it suffices to prove that $\E_0[\bar L_1] \ge 1 - o(1/n)$.
We develop
\begin{align}
\E_0[\bar L_1]
&= \E_0\bigg[\bigg(1 -\eps + \frac{\eps}{\sigma} \exp\bigg(\frac{\sigma^2 -1}{2 \sigma^2} X_1^2\bigg) \bigg) D_1\bigg] \\
&= (1 -\eps) (1-2\bar\Phi(\sqrt{2 \log n})) + \eps (1-2\bar\Phi(\sqrt{2 \log n}/\sigma)) \\
&= (1 -\eps) (1 - O(n^{-1}/\sqrt{\log n})) + \eps (1 - O(n^{-1/\sigma^2}/\sqrt{\log n})) \\
&= 1 - o(1/n) - o(\eps n^{-1/\sigma^2}),
\end{align}
where $\bar\Phi$ is the standard normal survival function.
We used the well-known fact that $\bar\Phi(t) \sim e^{-t^2/2}/\sqrt{2\pi} t$ as $t \to \infty$.
Since $\eps = n^{-\beta}$ with $\beta > 1/2$, and \eqref{LR1+} holds, we have $\eps n^{-1/\sigma^2} = o(1/n)$, so that $\E_0[\bar L_1] \ge 1 - o(1/n)$.

For the second moment, we have
\beq
\E_0[\bar L^2] 
= \prod_{i=1}^n \E_0[\bar L_i^2]
= \E_0[\bar L_1^2]^n,
\eeq
so that it suffices to prove that $\E_0[\bar L_1^2] \le 1 + o(1/n)$.
We develop
\begin{align}
\E_0[\bar L_1^2]
&= \E_0\bigg[\bigg(1 -\eps + \frac{\eps}{\sigma} \exp\bigg(\frac{\sigma^2 -1}{2 \sigma^2} X_1^2\bigg)\bigg)^2 D_1\bigg] \\
&= (1 -\eps)^2 (1-2\bar\Phi(\sqrt{2 \log n})) + 2 (1-\eps) \eps\, (1-2\bar\Phi(\sqrt{2 \log n}/\sigma))  + \frac{\eps^2}{\sigma^2} \E_0\bigg[\exp\bigg(\frac{\sigma^2-1}{\sigma^2} X_1^2\bigg) D_1\bigg]\\
&\le 1-\eps^2 + \frac{\eps^2}{\sqrt{2\pi}\sigma^2} \int_{-\sqrt{2 \log n}}^{\sqrt{2 \log n}} \exp\bigg(\bigg(\frac{\sigma^2-1}{\sigma^2} - \frac{1}{2}\bigg) x^2\bigg) {\rm d}x\\
&\le 1 + O\bigg(\eps^2 \exp\bigg(\frac{(\sigma^2-2)_+}{\sigma^2} \log n\bigg) \sqrt{\log n} \bigg).
\end{align}
Hence, it suffices that $-2\beta + (\sigma^2-2)_+/\sigma^2 < -1$, which is equivalent to \eqref{LR1+}.
\end{proof}

\section{Other tests}

Having studied the performance of the likelihood ratio test, we  now turn to studying the performance of the chi-squared test, the extremes test, and the higher criticism test.  These tests are more practical in that they do not require knowledge of the  parameters driving the alternative, $(\eps, \sigma)$, to be implemented.

\subsection{The chi-squared test}
The chi-squared test is the classical variance test.  It happens to only be asymptotically powerful in the dense regime when $\sigma$ is bounded away from 1.

\begin{prp}\label{prp:chi2}
For the testing problem \eqref{problem} with parameterization \eqref{eps}, the chi-squared test is asymptotically powerful when $\beta < 1/2$ and either $\sigma$ is bounded away from~1 or \eqref{sigma1} holds with $\gamma < 1/2 - \beta$.
The chi-squared test is asymptotically powerless when $\beta > 1/2$ and $\sigma$ is bounded.
\end{prp}

\begin{proof}
We divide the proof into the two regimes.

\medskip\noindent
{\em Dense regime ($\beta < 1/2$).}
We show that there is a chi-squared test that is asymptotically powerful when $\beta < 1/2$.  To do so, we use Chebyshev's inequality.
Under $\cH_0$, $W := \sum_{i=1}^n X_i^2$ has the chi-squared distribution with $n$ degrees of freedom.  But using only the fact that $\E_0(W)=n$ and $\Var_0(W)=2n$, by Chebyshev's inequality, we have 
\beq
\P_0(|W - n| \ge a_n \sqrt{n}) \to 0,
\eeq
for any sequence $(a_n)$ diverging to infinity.  
Under $\cH_1$, $\E_1(W) = n(1-\eps + \eps \sigma^2)$ and $\Var_1(W)=2n(1-\eps + \eps \sigma^4)$.  Note that $\Var_1(W) \le 2n$ eventually. By Chebyshev's inequality,
\beq
\P_1(|W - n(1-\eps + \eps \sigma^2)| \ge a_n \sqrt{n}) \to 0.
\eeq
We choose $a_n = \log n$ and consider the test with rejection region $\{|W - n| \ge a_n \sqrt{n}\}$.  This test is asymptotically powerful when, eventually, 
\beq
|n(1-\eps + \eps \sigma^2) - n| \ge 2 a_n \sqrt{n},
\eeq
meaning,
\beq
|\sigma^2 - 1|  \eps \sqrt{n} \ge 2 a_n.
\eeq
This is the case when $\beta < 1/2$ with no condition on $\sigma$ other than remaining bounded away from 1, and also when \eqref{sigma1} holds and $\gamma < 1/2 - \beta$.

\medskip\noindent
{\em Sparse regime ($\beta > 1/2$).}
To prove that the chi-squared procedure is asymptotically powerless when $\beta > 1/2$, we argue in terms of convergence in distribution rather than the simple bounding of moments.  
Under $\cH_0$, the usual Central Limit Theorem implies that $(W-n)/\sqrt{2n}$ converges weakly to the standard normal distribution.
%\beq
%\P_0(|W - n| \ge a \sqrt{2 n}) \to 2(1-\Phi(a)),
%\eeq
%for any $a \ge 0$, where $\Phi$ is the standard normal distribution function.
Under $\cH_1$, the same is true using the Lyapunov Central Limit Theorem for triangular arrays.  Indeed, even though the distribution of $X_1, \dots, X_n$ depends on $(n, \eps)$, uniformly
\beq
\frac{\sum_{i=1}^n \E_1\big[(X_i^2 - 1)^4\big]}{\Big(\sum_{i=1}^n \E_1\big[(X_i^2 - 1)^2\big]\Big)^2} 
= \frac{n \E_1\big[(X_1^2 - 1)^4\big]}{n^2 \Big(\E_1\big[(X_1^2 - 1)^2\big]\Big)^2} \asymp 1/n \to 0,
\eeq
so that $(W-\E_1(W))/\sqrt{\Var_1(W)}$ converges weakly to the standard normal distribution.
Since 
\beq
\frac{W-\E_1(W)}{\sqrt{\Var_1(W)}} = \bigg(\frac{W - n}{\sqrt{2n}} + \frac{n - \E_1(W)}{\sqrt{2n}}\bigg) \frac{\sqrt{2n}}{\sqrt{\Var_1(W)}},
\eeq
with
\beq
\E_1(W) 
= n(1-\eps + \eps \sigma^2)
= n + o(\sqrt{n}), \quad \text{(since $\beta > 1/2$)},
\eeq
and
\beq
\Var_1(W)
= \sum_{i=1}^n \E_1\big[(X_i^2 - 1)^2\big] 
= 2n(1-\eps + \eps \sigma^4)
\sim 2n, \quad \text{(since $\sigma$ is bounded)},
\eeq
it is also the case that $(W-n)/\sqrt{2n}$ converges weakly to the standard normal distribution.
Hence, there is no test based on $W$ that has any asymptotic power.
\end{proof}

\subsection{The extremes test}
The extremes test, as the name indicates, focuses on the extreme observations, disregarding the rest of the sample.  It happens to be suboptimal in the setting where $\sigma \to 0$, while it achieves the detection boundary in the sparse regime in the setting where $\sigma$ is fixed. 

\begin{prp} \label{prp:extremes}
For the testing problem \eqref{problem} with parameterization \eqref{eps} and \eqref{sigma0}, the extremes test is asymptotically powerful when $\gamma > \beta$ (and asymptotically powerless when $\gamma < \beta$).
If instead $\sigma > 0$ is fixed, the extremes test is asymptotically powerful when $\sigma > 1/\sqrt{1-\beta}$ (and asymptotically powerless when $\sigma < 1/\sqrt{1-\beta}$).
\end{prp}

\begin{proof}
Under $\cH_0$, for any $a_n \to \infty$, we have 
\begin{align}
\P_0\big(\min_i |X_i| \ge 1/n a_n\big) 
&= \big[\P_0\big(|X_1| \ge 1/n a_n\big)\big]^n\\
&= \big[2\bar\Phi(1/n a_n)\big]^n\\
&= \big[1 - O(1/n a_n)\big]^n 
\to 1.
\end{align}
Similarly, as is well-known,
\beq
\P_0\big(\max_i |X_i| \le \sqrt{2 \log n}\big) \to 1.
\eeq
We thus consider the test with rejection region $\{\min_i |X_i| \le 1/n\log n\} \cup \{\max_i |X_i| \ge \sqrt{2 \log n}\}$.

We now consider the alternative.  
We first consider the case where \eqref{sigma0} holds.  We focus on the main sub-case where, in addition, $\gamma < 1$.
Let $I \subset \{1, \dots, n\}$ index the contaminated observations, meaning those sampled from $\cN(0, \sigma^2)$.  In our mixture model, $|I|$ is binomial with parameters $(n, \eps)$.
Let $Z_1, \dots, Z_n$ be iid standard normal variables and set $b_n = \sigma n \log n$.
We have
\begin{align}
\P_1\big(\min_i |X_i| \le 1/n \log n\big) 
&\ge \P_1\big(\min_{i \in I} |X_i| \le 1/n \log n\big) \\
&= 1 -  \E\big[\P\big(\min_{i \in I} |Z_i| \ge 1/b_n \mid I\big)\big] \\
&= 1 - \E\big[(2\bar\Phi(1/b_n))^{|I|}\big] \\
&= 1 - \big[1 - \eps + \eps 2\bar\Phi(1/b_n)\big]^n.
\end{align}
Since we have assumed that $\gamma < 1$ in \eqref{sigma0}, we have $1/b_n \to 0$, and therefore 
\beq
2\bar\Phi(1/b_n) 
= 1 - \frac{2 + o(1)}{\sqrt{2\pi} b_n}.
\eeq
This in turn implies that 
\beq
\big[1 - \eps + \eps 2\bar\Phi(1/b_n)\big]^n 
= \bigg[1 - \frac{(2 + o(1)) \eps}{\sqrt{2\pi} b_n}\bigg]^n 
\to 0
\eeq 
when $n \eps/b_n \to \infty$, which is the case when $\gamma > \beta$. 

Assume instead that $\gamma < \beta$. 
Fix a level $\alpha \in (0,1)$ and consider the extremes test at that level.  Based on the same calculations, this test has rejection region $\{\min_i |X_i| \le c_n\} \cup \{\max_i |X_i| \ge d_n\}$, where $c_n$ and $d_n$ are defined by $[2\bar\Phi(c_n)]^n = 1 -\alpha/2$ and $[2\Phi(d_n) -1]^n = 1 -\alpha/2$, respectively.
Note that
\beq
c_n \sim - \sqrt{\pi/2}\, \log(1 -\alpha/2)/n, \quad 
d_n \sim \sqrt{2 \log n}.
\eeq
For the minimum, we have
\begin{align}
\P_1\big(\min_i |X_i| \le c_n\big) 
&\le \P_1\big(\min_{i \nin I} |X_i| \le c_n\big) + \P_1\big(\min_{i \in I} |X_i| \le c_n\big).
% \\
%&= \P_1\big(\min_{i \nin I} |X_i| \le c_n\big) + (1 -  \E\big[\P\big(\min_{i \in I} |Z_i| \ge 1/n \sigma a_n \mid I\big)\big]) \\
\end{align}
Let $Z_1, \dots, Z_n$ be iid standard normal variables.
Clearly,
\beq
\P_1\big(\min_{i \nin I} |X_i| \le c_n\big)
\le \P\big(\min_i |Z_i| \le c_n\big) = \alpha/2,
\eeq
and, as was derived above, 
\begin{align}
\P_1\big(\min_{i \in I} |X_i| \le c_n\big)
&= \P_1\big(\min_{i \in I} |Z_i| \le c_n/\sigma\big) \\
&= 1 - \big[1 - \eps + \eps 2\bar\Phi(c_n/\sigma)\big]^n,
\end{align}
with
\beq
\big[1 - \eps + \eps 2\bar\Phi(c_n/\sigma)\big]^n
= \bigg[1 - \frac{(2 + o(1)) \eps c_n}{\sqrt{2\pi} \sigma}\bigg]^n \to 1,
\eeq
since $\eps c_n/\sigma \asymp n^{-1 -\beta + \gamma} = o(1/n)$.
Thus, $\P_1(\min_i |X_i| \le c_n) \to 0$.  
And since $\max_i |X_i|$ under the alternative is stochastically bounded from above by its distribution under the null (since $\sigma < 1$), we also have $\P_1(\max_i |X_i| \ge d_n) \to 0$.
Hence, the extremes test (at level $\alpha$ arbitrary) has asymptotic power $\alpha$, meaning it is asymptotically powerless.  (It is no better than random guessing.)

Next, we consider the case where $\sigma$ is fixed.
Following similar arguments, now with $b_n = \sigma^{-1}\sqrt{2 \log n}$, we have
\begin{align}
\P_1\big(\max_i |X_i| \ge \sqrt{2 \log n}\big) 
&\ge \P_1\big(\max_{i \in I} |X_i| \ge \sqrt{2 \log n}\big) \\
&= 1 - \E\big[\P\big(\max_{i \in I} |Z_i| \le b_n \mid I\big)\big] \\
&= 1 - \E\big[(2 \Phi(b_n) - 1)^{|I|}\big] \\
&= 1- \big[1 - \eps + \eps (2\Phi(b_n) - 1)\big]^n.
\end{align}
We have
\beq
2 \Phi(b_n) - 1 \asymp 1 - o(n^{-1/\sigma^2}),
\eeq
so that
\beq
\big[1 - \eps + \eps (2 \Phi(b_n) - 1)\big]^n
\asymp \big[1 -o(\eps n^{-1/\sigma^2})\big]^n
\to 0
\eeq
when $n \eps n^{-1/\sigma^2} \to \infty$, which is the case when $\sigma > 1/\sqrt{1-\beta}$.

Using a similar line of arguments, it can also be shown that the test is asymptotically powerless when $\sigma < 1/\sqrt{1-\beta}$ is fixed.
\end{proof}

\subsection{The higher criticism test}
The higher criticism, which looks at the entire sample via excursions of its empirical process, happens to achieve the detection boundary in all regimes, and is thus (first-order) comparable to the likelihood ratio test while being adaptive to the model parameters.

\begin{prp} \label{prp:hc}
For the testing problem \eqref{problem} with parameterization \eqref{eps}, the higher criticism test is asymptotically powerful when either \eqref{sigma0} holds with $\gamma > 2\beta - 1$, or \eqref{sigma1} holds with $\gamma < 1/2 - \beta$, or $\sigma > 1/\sqrt{1 - \beta}$ is fixed, or $\beta < 1/2$ and $\sigma \ne 1$ is fixed.
\end{prp}

\begin{proof}
Let $\hc$ denote the higher criticism statistic \eqref{HC}.  \citet{jaeschke1979asymptotic} derived the asymptotic distribution of $\hc$ under the null, and this weak convergence result in particular implies that 
\beq
\P_0\big(\hc \ge \sqrt{3 \log \log n}\big) \to 0.
\eeq
For simplicity, because it is enough for our purposes, we consider the test with rejection region $\{\hc \ge \log n\}$.
Note that the test is asymptotically powerful if, under the alternative, there is $t_n \ge 0$ such that
\beq
\frac{\sqrt{n}\, |F_n(t_n) - \Psi(t_n)|}{\sqrt{\Psi(t_n)(1-\Psi(t_n))}} \ge \log n
\eeq
with probability tending to 1.
To establish this, we will apply Chebyshev's inequality.  Indeed, $F_n(t)$ is binomial with parameters $n$ and $\Lambda(t) := (1-\eps) \Psi(t) + \eps \Psi(t/\sigma)$, so that
\beq
\frac{\sqrt{n}\, |F_n(t_n) - \Lambda(t_n)|}{\sqrt{\Lambda(t_n) (1- \Lambda(t_n))}} \le \log n
\eeq
with probability tending to 1.
When this is the case, we have
\beq
\frac{\sqrt{n}\, |F_n(t_n) - \Psi(t_n)|}{\sqrt{\Psi(t_n)(1-\Psi(t_n))}} \ge u_n - (\log n) \sqrt{v_n},
\eeq
where
\beq
u_n := \frac{\sqrt{n}\, \eps\, |\Psi(t_n/\sigma) - \Psi(t_n)|}{\sqrt{\Psi(t_n)(1-\Psi(t_n))}}, \quad
v_n := \frac{\Lambda(t_n) (1- \Lambda(t_n))}{\Psi(t_n)(1-\Psi(t_n))},
\eeq
and only need to prove that
\beq\label{HC-power}
u_n \ge (\sqrt{v_n} + 1) \log n.
\eeq

First, assume that \eqref{sigma0} holds with $\gamma > 2\beta - 1$. We focus on the interesting sub-case where $\gamma < \beta$. Fix $q$ such that $q > \gamma$ and $1/2 - \beta - q/2 + \gamma > 0$ and set $t_n = n^{-q}$.  Then, using the fact that $\eps/\sigma = n^{\gamma - \beta} = o(1)$, we have 
\beq
\Psi(t_n) \asymp t_n, \quad
\Psi(t_n/\sigma) \asymp t_n/\sigma, \quad
\Lambda(t_n) \asymp t_n + \eps t_n/\sigma \asymp t_n,
\eeq
so that
\beq
u_n \asymp \sqrt{n} \eps (t_n/\sigma)/\sqrt{t_n} = n^{1/2 - \beta - q/2 + \gamma} \gg \log n, \quad
v_n \asymp 1,
\eeq
and therefore \eqref{HC-power} is fulfilled, eventually.

Next, we assume that \eqref{sigma1} holds with $\gamma < 1/2 - \beta$.  Here we set $t_n = 1$, and get $0 < \Psi(t_n) =\Psi(1) < 1$, and 
\beq
|\Psi(t_n/\sigma) - \Psi(t_n)| \sim |(1/\sigma - 1) \Psi'(1)| \asymp |\sigma - 1|, \quad
\Lambda(t_n) \asymp 1,
\eeq
so that
\beq
u_n \asymp \sqrt{n} \eps |\sigma - 1| = n^{1/2 - \beta -\gamma} \gg \log n, \quad
v_n \asymp 1,
\eeq
and therefore \eqref{HC-power} is fulfilled, eventually.

The same arguments apply to the case where $\beta < 1/2$ and $\sigma \ne 1$ is fixed.  (It essentially corresponds to the previous case with $\gamma = 0$.)

The remaining case is where $\sigma > 1/\sqrt{1 - \beta}$ is fixed, with $\beta > 1/2$ (for otherwise it is included in the previous case).  We choose $t_n = \sqrt{2 q \log n}$, with $q := \beta/(1-1/\sigma^2)$, and get
%, motivated by the fact that we know that the extremes test achieves this boundary (\prpref{extremes}), and this choice of threshold makes the higher criticism behave like the extremes test.
\beq
t_n \bar\Psi(t_n) \asymp e^{-t_n^2/2} = n^{-q}, \quad
t_n \bar\Psi(t_n/\sigma) \asymp e^{-t_n^2/2\sigma^2} = n^{-q/\sigma^2},
\eeq
and
\beq
t_n \bar\Lambda(t_n) 
\asymp e^{-t_n^2/2} + \eps e^{-t_n^2/2\sigma^2}
= n^{-q} + n^{-\beta-q/\sigma^2} 
= 2 n^{-q},
\eeq
so that
\beq
u_n \asymp \frac{\sqrt{n} \eps e^{-t_n^2/2\sigma^2}/t_n}{\sqrt{e^{-t_n^2/2}/t_n}} \asymp n^{1/2 - \beta - q/\sigma^2 + q/2}/(\log n)^{1/4} \gg \log n, \quad
v_n \asymp 1,
\eeq
and therefore \eqref{HC-power} is fulfilled, eventually.
\end{proof}

\section{Some numerical experiments}

We performed some numerical experiments to investigate the finite sample performance of the tests considered here: the likelihood ratio test, the chi-squared test, the extremes test, the higher criticism test. 
The sample size $n$ was set large to $10^5$ in order to capture the large-sample behavior of these tests. We tried four scenarios with different combinations of $(\beta, \sigma)$. The p-values for each test are calibrated as follows:
\renewcommand{\theenumi}{(\alph{enumi})}
\renewcommand{\labelenumi}{\theenumi}
\benum  \setlength{\itemsep}{0in}
\item For the {\em likelihood ratio test} and the {\em higher criticism test}, we simulated the null distribution based on $10^4$  Monte Carlo replicates. 
\item For the {\em extremes test} and the {\em chi-squared test}, we used the exact null distribution, which in each case is available in closed form.
\eenum

For each combination of $(\beta, \sigma)$, we repeated the whole process 200 times and recorded the fraction of p-values  smaller than 0.05, representing the empirical power at the 0.05 level.  The result of this experiment is reported in \figref{numerics} and is largely congruent with the theory developed earlier in the paper.

\begin{figure}[!thpb]
\centering
\includegraphics[width=0.95\textwidth]{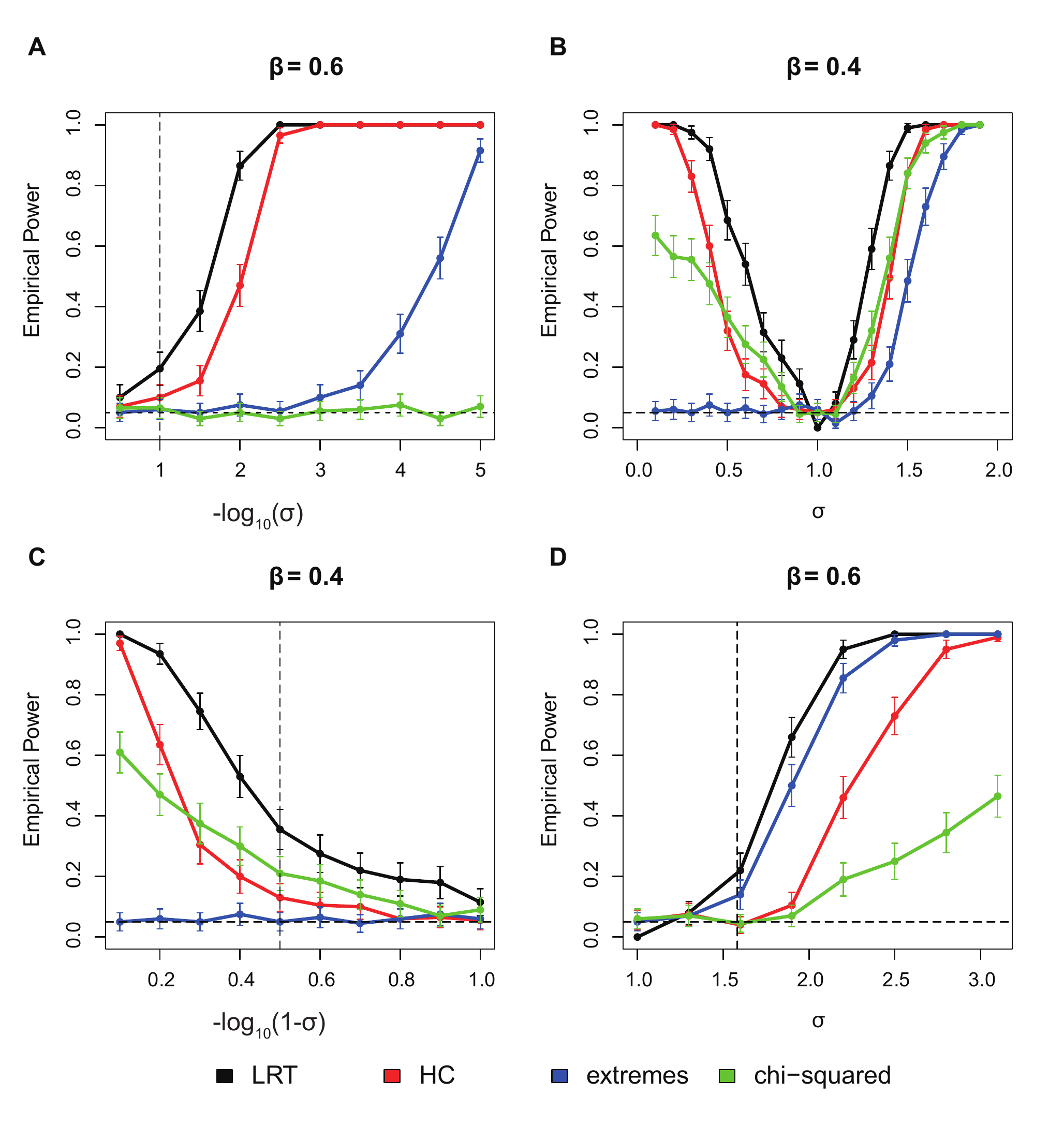}
\caption{Empirical power comparison with 95\% error bars. 
A. Sparse regime where $\beta = 0.6$ and $\sigma \to 0$. 
B. Dense regime where $\beta = 0.4$ and $\sigma$ fixed.  Note that the LR test is here asymptotically powerful at any $\sigma \ne 1$. 
C. Dense regime where $\beta = 0.4$ and $\sigma \to 1$. 
D. Sparse regime where $\beta = 0.6$ and $\sigma > 1$.  
Each time, the horizontal line marks the level (set at 0.05) and the vertical line marks the asymptotic detection boundary derived earlier in the paper.}
\label{fig:numerics}
\end{figure}

%\subsection*{Acknowledgments}
%We are grateful to XXX.  This work was partially supported by XXX.

\bibliographystyle{abbrvnat}
\bibliography{ref}

\end{document}

%% file: notation.tex
\def\hc{H}

\newtheorem{thm}{Theorem}
\newtheorem{prp}{Proposition}

\theoremstyle{remark}

\def\beq{\begin{equation}} % \setcounter{equation}{1}}
\def\eeq{\end{equation}}
\def\beqn{\begin{eqnarray*}}
\def\eeqn{\end{eqnarray*}}
\def\Bitem{\begin{itemize}\setlength{\itemsep}{.2in}}
\def\bitem{\begin{itemize}\setlength{\itemsep}{.05in}}
\def\eitem{\end{itemize}}
\def\Benum{\begin{enumerate}\setlength{\itemsep}{.2in}}
\def\benum{\begin{enumerate}\setlength{\itemsep}{.05in}}
\def\eenum{\end{enumerate}}
\def\bmult{\begin{multline*}}
\def\emult{\end{multline*}}
\def\bcenter{\begin{center}}
\def\ecenter{\end{center}}
\def\bframe{\begin{frame}}
\def\eframe{\end{frame}}

\newcommand{\thmref}[1]{Theorem~\ref{thm:#1}}

\newcommand{\figref}[1]{Figure~\ref{fig:#1}}

\def\cH{\mathcal{H}}

\def\cN{\mathcal{N}}

\def\bbI{\mathbb{I}}

\newcommand{\E}{\operatorname{\mathbb{E}}}
\renewcommand{\P}{\operatorname{\mathbb{P}}}
\newcommand{\Var}{\operatorname{Var}}

\def\eps{\varepsilon}

\def\1{\mathbbm{1}}
\newcommand{\IND}[1]{\bbI\{ #1 \}}